\documentclass{amsart}
\usepackage[utf8]{inputenc}
\usepackage[T1]{fontenc}
\usepackage{lmodern}
\usepackage[english]{babel}

\usepackage[reftex]{theoremref}
\usepackage{enumerate}
\usepackage{amsrefs}

\swapnumbers

\theoremstyle{plain}
	\newtheorem{theorem}{Theorem}[section]
	\newtheorem*{theorem*}{Theorem}
	\newtheorem{lemma}[theorem]{Lemma}
	\newtheorem{corollary}[theorem]{Corollary}
	
\theoremstyle{definition}
	
	\newtheorem{example}[theorem]{Example}
	
\theoremstyle{remark}
	\newtheorem{remark}[theorem]{Remark}

\newcommand{\cK}	{\mathcal K}

\newcommand{\cN}	{\mathcal N}
\newcommand{\cP}	{\mathcal P}
\newcommand{\cQ}	{\mathcal Q}
\newcommand{\Cp}	{C_p(\R^n)}
\newcommand{\gln}	{\operatorname{GL}(n)}
\newcommand{\h}[2]	{h\left( #1 , #2 \right)}
\newcommand{\Lp}	{L_p}
\newcommand{\N}		{\mathbb N}
\newcommand{\oDeltapp}	{\operatorname{\Delta}_p^+}
\newcommand{\oDeltapm}	{\operatorname{\Delta}_p^-}

\newcommand{\oPipp}	{\operatorname{\Pi}_p^+}
\newcommand{\oPippg}	{\operatorname{\Pi}_p^{+>}}
\newcommand{\oPippl}	{\operatorname{\Pi}_p^{+<}}
\newcommand{\oPipm}	{\operatorname{\Pi}_p^-}
\newcommand{\oPipmg}	{\operatorname{\Pi}_p^{->}}
\newcommand{\oPipml}	{\operatorname{\Pi}_p^{-<}}
\newcommand{\oPhi}	{\operatorname{\Phi}}
\newcommand{\oPsi}	{\operatorname{\Psi}}
\newcommand{\pp}	{+_p}

\newcommand{\R}		{\mathbb R}

\newcommand{\sln}	{\operatorname{SL}(n)}

\DeclareMathOperator{\vol}{vol}

\DeclareMathOperator{\aff}{aff}

\DeclareMathOperator{\conv}{conv}
\DeclareMathOperator{\lin}{lin}

\DeclareMathOperator{\sgn}{sgn}

\begin{document}
	\title{$\sln$-Contravariant $\Lp$-Minkowski Valuations}
	\author{Lukas Parapatits}
	\address
	{
		Fachbereich Mathematik \\
		Universität Salzburg \\
		Hellbrunner Str. 34 \\
		5020 Salzburg \\
		Austria
	}
	\email{lukas.parapatits@sbg.ac.at}
	\thanks{The author was supported by Austrian Science Fund (FWF): P22388 and Austrian Science Fund (FWF): P23639}

	\begin{abstract}
		All $\sln$-contravariant $\Lp$-Minkowski valuations on polytopes are completely classified.
		The prototypes of such valuations turn out to be the asymmetric $\Lp$-projection body operators.
	\end{abstract}

	\maketitle

	\section{Introduction}
		%%%%%%%%%%%%%%%%%%%%%%%%%%%%%%%%%%%%%%%%%%%%%%%%%%%%%%%%%%%%%%%%%%%%%%%%%%%%%%%%
A map $\oPhi$ defined on the set of convex bodies, i.e. nonempty compact convex subsets of $\R^n$,
is called a valuation, if it satisfies
\begin{equation*}
	\oPhi( K \cup L ) + \oPhi( K \cap L ) = \oPhi( K ) + \oPhi( L )
\end{equation*}
for all convex bodies $K,L$ such that $K \cup L$ is also convex.
Ever since Hadwiger's famous characterization theorem (see e.g. \cite{Schneider1993}), which describes all rigid motion invariant scalar valuations,
classifications of valuations have become the focus of intense research
(see e.g. \cites{Alesker1999, Alesker2000, Alesker 2001, Bernig2009, BernigBroecker2007, BernigFu_1,
Klain1996, Klain1997, Klain2000, LudwigReitzner1999, LudwigReitzner2010, McMullen1993}).
In the last few years a theory of convex body valued valuations emerged
(see e.g. \cites{AbardiaBernig2011, Haberl2008, Haberl2009, Haberl_1, Haberl_2, Ludwig2002_1, Ludwig2002_2, Ludwig2003, Ludwig2005, Ludwig2006, Ludwig2010, ParapatitsSchuster_1, SchneiderSchuster2006, Schuster2010, SchusterWannerer_1, Wannerer_1}).
One important type of such valuations are Minkowski valuations, i.e. valuations with respect to Minkowski addition,
which is given by $K + L = \{x+y : x \in K, y \in L\}$.
A prominent example of such a valuation is the projection body operator, which was introduced by Minkowski at the turn of the previous century (see e.g. \cite{Schneider1993}).
The projection body of a convex body $K$ encodes the information obtained from the $(n-1)$-dimensional volumes of projections of $K$
onto $(n-1)$-dimensional subspaces into a single convex body.
Projection bodies and their generalizations found applications in many areas such as convexity, stochastic geometry, functional analysis and geometric tomography
(see e.g. \cites{AbardiaBernig2011, CianchiLutwakYangZhang2009, HaberlSchuster2009_1, HaberlSchuster2009_2, Gardner2006,
LutwakYangZhang2000_1, LutwakYangZhang2002_1, LutwakYangZhang2002_2, LutwakYangZhang2004, LutwakYangZhang2010,
SchneiderWeil2008, SchusterWannerer_1, Steineder2008, Zhang1999}).

In two far reaching articles Ludwig \cites{Ludwig2002_1, Ludwig2005} characterized the projection body operator up to a constant
as the only continuous and homogeneous Minkowski valuation, which is $\sln$-contravariant.
A Minkowski valuation $\oPhi$ is called $\sln$-contravariant, if it satisfies
\begin{equation*}
	\oPhi( \phi K ) = \phi^{-t} \oPhi K
\end{equation*}
for all convex bodies $K$ and all $\phi$ in the special linear group $\sln$.
Moreover, Ludwig proved similar characterizations for Minkowski valuations defined on polytopes without continuity assumptions.
Very recently, Haberl \cite{Haberl_2} showed that the homogeneity assumptions in Ludwig's characterization results are not necessary.

Another important class of convex body valued valuations are $\Lp$-Minkowski valuations for $p > 1$.
They are valuations with respect to $\Lp$-Minkowski addition, which is given by
\begin{equation*}
	\h{K \pp L}{.}^p = \h{K}{.}^p + \h{L}{.}^p ,
\end{equation*}
where $\h{K}{u} := \max_{x \in K} \langle x,u \rangle$ denotes the support function of $K$.
This addition is the basis of the $\Lp$-Brunn-Minkowski theory
(see e.g. \cites{CampiGronchi2002, LutwakYangZhang2000_2, HaberlLudwig2006, HaberlSchuster2009_2,
Ludwig2005, LutwakYangZhang2000_1, LutwakYangZhang2002_1, LutwakYangZhang2002_2, LutwakYangZhang2004,
Stancu2002, Stancu2003, WernerYe2008}).
Important operators in this theory are the symmetric and asymmetric $\Lp$ analogs of the projection body.
The symmetric $\Lp$-projection body was first introduced by Lutwak, Yang and Zhang in \cite{LutwakYangZhang2000_1}.
Volume inequalities for symmetric and asymmetric $\Lp$-projection bodies are the geometric core of affine Sobolev inequalities,
which turn out to be stronger than their classical counterparts
(see e.g. \cites{CianchiLutwakYangZhang2009, HaberlSchuster2009_1, LutwakYangZhang2002_2, Zhang1999}).
In \cite{Ludwig2005} Ludwig also extended her characterization of the projection body operator to the $\Lp$ case.

In this article we remove the assumption of homogeneity in this characterization of the $\Lp$-projection body operators.
The set of convex polytopes (respectively convex bodies) containing the origin is denoted by $\cP^n_o$ (respectively $\cK^n_o$).
The asymmetric $\Lp$-projection body operator $\oPipp \colon \cP^n_o \to \cK^n_o$ is defined by
\begin{equation*}
	\h{ \oPipp P}{.}^p = \sum_{\substack{v \in \cN(P) \\ 0 \not\in F(P,v)}} \vol_{n-1}(F(P,v)) \h{P}{v}^{1-p} \langle v, . \rangle_+^p
\end{equation*}
for all $P \in \cP^n_o$.
Here $\cN(P)$ denotes the set of all outer unit normals of facets of $P$ and $F(P,v)$ denotes the facet corresponding to $v \in \cN(P)$.
Moreover $\langle v, . \rangle_+$ denotes the positive part of $\langle v, . \rangle$.
Similarly we define $\oPipm$ by replacing $\langle v, . \rangle_+$ with $\langle v, . \rangle_-$,
which denotes the negative part of $\langle v, . \rangle$.

\begin{theorem*}
	An operator $\oPhi \colon \cP^n_o \rightarrow \cK^n_o$, where $n \geq 3$, is an $\sln$-contravariant $\Lp$-Minkowski valuation,
	if and only if there exist constants $c_1, c_2 \geq 0$ such that
	\begin{equation*}
		\oPhi P = c_1 \oPipp P \pp c_2 \oPipm P
	\end{equation*}
	for all $P \in \cP^n_o$.
\end{theorem*}

All the above characterizations deal with valuations defined on convex polytopes (or convex bodies) containing the origin.
Building on Ludwig's characterization of the projection body operator Schuster and Wannerer \cite{SchusterWannerer_1}
obtained a classification of Minkowski valuations on the set of \textit{all} convex bodies
and showed that an additional operator arises.
Our second main theorem establishes a corresponding classification for $\Lp$-Minkowski valuations
defined on the set of \textit{all} convex polytopes, denoted by $\cP^n$.
As it turns out, also in this case new operators arise.
For the definitions of the operators $\oPippg$, $\oPipmg$, $\oPippl$ and $\oPipml$ see Section \ref{se: prelimSLn}.

\begin{theorem*}
	An operator $\oPhi \colon \cP^n \rightarrow \cK^n_o$, where $n \geq 3$, is an $\sln$-contravariant $\Lp$-Minkowski valuation,
	if and only if there exist constants $c_1, c_2, c_3, c_4 \geq 0$ such that
	\begin{equation*}
		\oPhi P = c_1 \oPippg P \pp c_2 \oPipmg P \pp c_3 \oPippl P \pp c_4 \oPipml P
	\end{equation*}
	for all $P \in \cP^n$.
\end{theorem*}

	\section{Preliminaries}
		%%%%%%%%%%%%%%%%%%%%%%%%%%%%%%%%%%%%%%%%%%%%%%%%%%%%%%%%%%%%%%%%%%%%%%%%%%%%%%%%
As a general reference for the concepts introduced in the following sections see \cites{Gruber2007, KlainRota1997, Schneider1993}.
Throughout this article $n \geq 1$ will denote the dimension of the Euclidean space $\R^n$.
The vectors $e_1, \ldots, e_n$ are the standard basis vectors.
We denote by $\langle x , y \rangle$ the inner product of two vectors $x,y \in \R^n$.
The orthogonal complement of $x$ is denoted by $x^\perp$ and the norm induced by the inner product is denoted by $\| x \|$.
The unit ball in $\R^n$ is written as $B^n$ and its boundary as $S^{n-1}$.
The $m$-dimensional volume in an $m$-dimensional subspace will be written as $\vol_m$ for $1 \leq m \leq n$.
The general linear group and the special linear group are denoted by $\gln$ and $\sln$, respectively.
We denote by $\lin$ the linear hull and by $\conv$ the convex hull of a subset of $\R^n$.

%%%%%%%%%%%%%%%%%%%%%%%%%%%%%%%%%%%%%%%%%%%%%%%%%%%%%%%%%%%%%%%%%%%%%%%%%%%%%%%%
A nonempty compact convex subset of $\R^n$ is called a convex body.
The set of all convex bodies in $\R^n$ is denoted by $\cK^n$.
We denote by $\dim$ the dimension of a convex body.
The convex hull of a finite set of points in $\R^n$ is called a convex polytope.
The set of all convex polytopes in $\R^n$ is denoted by $\cP^n$.
Clearly $\cP^n \subseteq \cK^n$.
The subset of all convex bodies and convex polytopes containing the origin $o$ are denoted by $\cK^n_o$ and $\cP^n_o$, respectively.

%%%%%%%%%%%%%%%%%%%%%%%%%%%%%%%%%%%%%%%%%%%%%%%%%%%%%%%%%%%%%%%%%%%%%%%%%%%%%%%%
The Minkowski sum of two convex bodies $K$ and $L$, denoted $K+L$, is defined by
\begin{equation*}
	K + L = \{x+y : x \in K, y \in L\} .
\end{equation*}
The scalar multiple of a convex body $K$ and $s \geq 0$, denoted $sK$, is defined by
\begin{equation*}
	sK = \{sx : x \in K\} .
\end{equation*}
Note that $\cK^n$, $\cP^n$, $\cK^n_o$ and $\cP^n_o$ are closed under these operations.
We equip $\cK^n$ with the Hausdorff metric $\delta$, defined by
\begin{equation*}
	\delta(K,L) = \min\{ \epsilon > 0 : K + \epsilon B^n \subseteq L, L + \epsilon B^n \subseteq K \}
\end{equation*}
for all $K,L \in \cK^n$.
Note that the above operations are continuous and that $\cK^n$ is a complete metric space.
Furthermore $\cP^n$, $\cK^n_o$ and $\cP^n_o$ are closed sets in this topology.

%%%%%%%%%%%%%%%%%%%%%%%%%%%%%%%%%%%%%%%%%%%%%%%%%%%%%%%%%%%%%%%%%%%%%%%%%%%%%%%%
Every $K \in \cK^n$ is characterised by its support function
\begin{equation*}
	\h{K}{u} := \max_{x \in K} \langle x,u \rangle \quad u \in \R^n .
\end{equation*}
A function $h \colon \R^n \to \R$ is a support function, if and only if it is sublinear, i.e.
\begin{equation*}
	h(u+v) \leq h(u) + h(v) \quad \text{and} \quad h(su) = sh(u)
\end{equation*}
for all $u,v \in \R^n$ and $s > 0$.
Note that sublinearity implies convexity and therefore continuity.
Because of the homogeneity, a support function is determined by its values on $S^{n-1}$.
Minkowski addition and scalar multiplication are compatible with the map $K \mapsto \h{K}{.}$, i.e.
\begin{equation*}
	\h{K+L}{.} = \h{K}{.} + \h{L}{.} \quad \text{and} \quad \h{sK}{.} = s\h{K}{.}
\end{equation*}
for all $K,L \in \cK^n$ and $s \geq 0$.
The Hausdorff distance of two convex bodies $K,L \in \cK^n$ can be calculated by
\begin{equation*}
	\delta(K,L) = \| \h{K}{.} - \h{L}{.} \|_\infty ,
\end{equation*}
where $\| . \|_\infty$ denotes the infinity norm on $S^{n-1}$.
In particular, we can think of $\cK^n$ with Minkowski addition as a subsemigroup of the abelian group $C(\R^n)$.

%%%%%%%%%%%%%%%%%%%%%%%%%%%%%%%%%%%%%%%%%%%%%%%%%%%%%%%%%%%%%%%%%%%%%%%%%%%%%%%%
Throughout this article $p > 1$ will denote a real number.
Note that a convex body $K$ contains the origin, if and only if $\h{K}{.} \geq 0$.
It is easy to see that $\sqrt[p]{\h{K}{.}^p + \h{L}{.}^p}$ defines a non-negative sublinear function for all $K,L \in \cK^n_o$.
It is therefore the support function of a unique convex body in $\cK^n_o$.
The $\Lp$-Minkowski sum of $K,L \in \cK^n_o$, denoted $K \pp L$, is defined by
\begin{equation*}
	\h{K \pp L}{.}^p = \h{K}{.}^p + \h{L}{.}^p .
\end{equation*}
By identifying $K \in \cK^n_o$ with $\h{K}{.}^p$ we can think of $\cK^n_o$ with $\Lp$-Minkowski addition as a subsemigroup of $C(\R^n)$.
Clearly $\h{K}{.}^p$ is a $p$-homogeneous function for all $K \in \cK^n_o$.
We denote by $\Cp$ the set of all $p$-homogeneous functions in $C(\R^n)$.

%%%%%%%%%%%%%%%%%%%%%%%%%%%%%%%%%%%%%%%%%%%%%%%%%%%%%%%%%%%%%%%%%%%%%%%%%%%%%%%%
Cauchy's functional equation
\begin{equation}\label{cauchy}
	f(a+b) = f(a) + f(b) \quad \forall a,b \in \R
\end{equation}
will be important for us.
Let $f \colon \R \to \R$ be a non-linear function which satisfies \eqref{cauchy}.
It is a well known fact that the graph of such a function $f$ is dense in $\R^2$.
An equivalent statement is that every bounded open interval has a dense image under $f$.

Let $f \colon (0,+\infty) \to \R$ be a non-linear function which satisfies \eqref{cauchy} for all $a,b \in (0,+\infty)$.
It is easy to see that we can extend $f$ to an odd function $f \colon \R \to \R$ which satisfies \eqref{cauchy} for all $a,b \in \R$.
Therefore every bounded open interval which is a subset of $(0,+\infty)$ has a dense image under $f$.

	\section{Valuations}
		%%%%%%%%%%%%%%%%%%%%%%%%%%%%%%%%%%%%%%%%%%%%%%%%%%%%%%%%%%%%%%%%%%%%%%%%%%%%%%%%
Let $\cQ^n$ be a subset of $\cK^n$ and let $A$ be an abelian semigroup.
A map $\oPhi \colon \cQ^n \to A$ is called a valuation, if it satisfies
\begin{equation*}
	\oPhi( K \cup L ) + \oPhi( K \cap L ) = \oPhi( K ) + \oPhi( L )
\end{equation*}
for all $K,L \in \cQ^n$ with $K \cup L, K \cap L \in \cQ^n$.
Furthermore, if $A$ has an identity $0$, we assume $\Phi(\emptyset) = 0$, even if $\emptyset \not\in \cQ^n$.
If $A$ is $\cK^n$ with Minkowski addition, then $\oPhi$ is called a Minkowski valuation.
Note that $\oPhi$ is a Minkowski valuation, if and only if $K \mapsto \h{\oPhi K}{.} \in C(\R^n)$ is a valuation.
If $A$ is $\cK^n_o$ with $\Lp$-Minkowski addition, then $\oPhi$ is called an $\Lp$-Minkowski valuation.
Note that $\oPhi$ is an $\Lp$-Minkowski valuation, if and only if $K \mapsto \h{\oPhi K}{.}^p \in C(\R^n)$ is a valuation.
It is easy to see that $K \cup L \in \cK^n$ implies $K \cap L \in \cK^n$ for all $K,L \in \cK^n$.
The same holds for $\cK^n_o$.

Let $A$ be an abelian monoid with identity $0$.
A valuation $\oPhi \colon \cQ^n \to A$ is called simple, if $\oPhi K = 0$ for all $K \in \cQ^n$ with $\dim K < n$.

%%%%%%%%%%%%%%%%%%%%%%%%%%%%%%%%%%%%%%%%%%%%%%%%%%%%%%%%%%%%%%%%%%%%%%%%%%%%%%%%
A $k$-dimensional simplex is the convex hull of $k+1$ affinely independent points for $k \in \{0,\ldots,n\}$.
The $n$-dimensional standard simplex, denoted $T^n$, is defined by
\begin{equation*}
	T^n = \conv\{ o, e_1, \ldots, e_n \} .
\end{equation*}

We need some general theorems on valuations.
With the exception of the next theorem due to Volland \cite{Volland1957} (see also \cite{KlainRota1997}) we give proofs for the sake of completeness.

\begin{theorem}\th\label{incl excl Pn}
	Let $A$ be an abelian group and $\oPhi \colon \cP^n \to A$ a valuation.
	Then $\oPhi$ satisfies the inclusion exclusion principle, i.e.
	\begin{equation*}
		\oPhi(P_1 \cup \ldots \cup P_m) =
		\sum_{\emptyset \neq S \subseteq \{1,\ldots,m\}} (-1)^{|S|-1} \oPhi \left( \bigcap_{i \in S} P_i \right)
	\end{equation*}
	for all $m \in \N$ and $P_1, \ldots, P_m \in \cP^n$ with $P_1 \cup \ldots \cup P_m \in \cP^n$.
\end{theorem}

\begin{lemma}\th\label{valuation determined simplices Pn}
	Let $A$ be an abelian group and $\oPhi \colon \cP^n \to A$ a valuation.
	Then $\oPhi$ is determined by its values on $n$-simplices.
\end{lemma}
\begin{proof}
	Assume that $\oPhi$ vanishes on $n$-simplices.
	Let $k \in \{0,\ldots,n-1\}$.
	Every $k$-dimensional simplex $T$ can be written as the intersection of two $(k+1)$-dimensional simplices $T_1, T_2$.
	We can do this in such a way that $T_1 \cup T_2$ is also a $(k+1)$-dimensional simplex.
	Using induction from $n$ to $0$ shows that $\oPhi$ vanishes on all simplices.
	Now, using \th\ref{incl excl Pn} and induction from $0$ to $n$ finishes the proof.
\end{proof}

\begin{lemma}\th\label{incl excl Pno}
	Let $A$ be an abelian group and $\oPhi \colon \cP^n_o \to A$ a valuation.
	Then $\oPhi$ satisfies the inclusion exclusion principle, i.e.
	\begin{equation*}
		\oPhi(P_1 \cup \ldots \cup P_m) =
		\sum_{\emptyset \neq S \subseteq \{1,\ldots,m\}} (-1)^{|S|-1} \oPhi \left( \bigcap_{i \in S} P_i \right)
	\end{equation*}
	for all $m \in \N$ and $P_1, \ldots, P_m \in \cP^n_o$ with $P_1 \cup \ldots \cup P_m \in \cP^n_o$.
\end{lemma}
\begin{proof}
	Extend $\oPhi$ to $\cP^n$ by
	\begin{equation*}
		\oPhi P = \oPhi (P_o)
	\end{equation*}
	for all $P \in \cP^n$, where $P_o := \conv(\{0\} \cup P)$.
	It is easy to see that this defines a valuation on $\cP^n$.
	The assertion now follows from \th\ref{incl excl Pn}.
\end{proof}

\begin{lemma}\th\label{valuation determined simplices}
	Let $A$ be an abelian group and $\oPhi \colon \cP^n_o \to A$ a valuation.
	Then $\oPhi$ is determined by its values on $n$-simplices with one vertex at the origin
	and its value on $\{o\}$.
\end{lemma}
\begin{proof}
	Assume that $\oPhi$ vanishes on $n$-simplices with one vertex at the origin.
	Similar to the proof of \th\ref{valuation determined simplices Pn}, it follows that $\oPhi$ vanishes on all simplices with one vertex at the origin.
	Note that a $k$-dimensional polytope which contains the origin can be dissected into $k$-simplices with one vertex at the origin for all $k \in \{0,\ldots,n\}$.
	Now, using \th\ref{incl excl Pno} and induction from $0$ to $n$ finishes the proof.
\end{proof}

\begin{lemma}\th\label{simple determined cPno}
	Let $A$ be an abelian group and $\oPhi \colon \cP^n \to A$ a simple valuation.
	Then $\oPhi$ is determined by its values on $\cP^n_o$.
\end{lemma}
\begin{proof}
	Let $P \in \cP^n$.
	Denote by $F_1,\ldots,F_m$ the facets of $P$ which face towards the origin.
	Here we say that a facet $F$ is facing towards the origin, if $\h{P}{v} < 0$,
	where $v$ is the corresponding outer unit normal.
	Clearly
	\begin{equation*}
		P_o = P \cup (F_1)_o \cup \ldots \cup (F_m)_o ,
	\end{equation*}
	where $P_o := \conv(\{0\} \cup P)$.
	Note that $P_o, (F_1)_o, \ldots, (F_m)_o \in \cP^n_o$.
	Furthermore the intersection of two convex bodies of the right hand side is lower dimensional.
	Since $\oPhi$ is simple, \th\ref{incl excl Pn} implies
	\begin{equation*}
		\oPhi P_o = \oPhi P + \oPhi (F_1)_o + \ldots + \oPhi (F_m)_o
	\end{equation*}
	or equivalently
	\begin{equation*}
		\oPhi P = \oPhi P_o - \oPhi (F_1)_o - \ldots - \oPhi (F_m)_o .
	\end{equation*}
\end{proof}

	\section{$\sln$-Contravariance}
		\label{se: prelimSLn}
		%%%%%%%%%%%%%%%%%%%%%%%%%%%%%%%%%%%%%%%%%%%%%%%%%%%%%%%%%%%%%%%%%%%%%%%%%%%%%%%%
Let $\cQ^n$ denote $\cK^n$,$\cP^n$,$\cK^n_o$ or $\cP^n_o$.
A map $\oPhi \colon \cQ^n \to \cK^n$ is called $\sln$-contravariant, if it satisfies
\begin{equation*}
	\oPhi( \phi K ) = \phi^{-t} \oPhi K
\end{equation*}
for all $K \in \cQ^n$ and $\phi \in \sln$.
A map $\oPhi \colon \cQ^n \to C(\R^n)$ is called $\sln$-contravariant, if it satisfies
\begin{equation*}
	\oPhi( \phi K ) = \oPhi(K) \circ \phi^{-1}
\end{equation*}
for all $K \in \cQ^n$ and $\phi \in \sln$.
Since
\begin{equation*}
	\h{\phi^{-t} \oPhi K}{u} = \h{\oPhi K}{\phi^{-1} u}
\end{equation*}
holds for all $K \in \cQ^n$, $u \in \R^n$ and $\phi \in \sln$,
we see that a map $\oPhi \colon \cQ^n \to \cK^n$ is $\sln$-contravariant,
if and only if $K \mapsto \h{\oPhi K}{.}$ is $\sln$-contravariant.
Similarly a map $\oPhi \colon \cQ^n \to \cK^n_o$ is $\sln$-contravariant,
if and only if $K \mapsto \h{\oPhi K}{.}^p$ is $\sln$-contravariant.

%%%%%%%%%%%%%%%%%%%%%%%%%%%%%%%%%%%%%%%%%%%%%%%%%%%%%%%%%%%%%%%%%%%%%%%%%%%%%%%%
We will now define some $\sln$-contravariant $\Lp$-Minkowski valuations.
The $\sln$-contravariance and the valuation property will be proven below.
The asymmetric $\Lp$-projection body operator $\oPipp \colon \cP^n_o \to \cK^n_o$ is defined by
\begin{equation*}
	\h{ \oPipp P}{.}^p = \sum_{\substack{v \in \cN(P) \\ 0 \not\in F(P,v)}} \vol_{n-1}(F(P,v)) \h{P}{v}^{1-p} \langle v, . \rangle_+^p
\end{equation*}
for all $P \in \cP^n_o$.
Here $\cN(P)$ denotes the set of all outer unit normals of facets of $P$ and $F(P,v)$ denotes the facet corresponding to $v \in \cN(P)$.
More generally we define $F(P,v) = P \cap \{x \in \R^n : \langle x ,v \rangle = \h{P}{v}\}$.
Furthermore $\langle v, . \rangle_+$ denotes the positive part of $\langle v, . \rangle$, i.e. $\max\{\langle v, . \rangle, 0\}$.
Note that $\langle v, . \rangle_+$ is the support function of the straight line segment $[o, v]$.
The map $\oPippg \colon \cP^n \to \cK^n_o$ is defined by
\begin{equation*}
	\h{ \oPippg P}{.}^p = \sum_{\substack{v \in \cN(P) \\ \h{P}{v} > 0}} \vol_{n-1}(F(P,v)) \h{P}{v}^{1-p} \langle v, . \rangle_+^p
\end{equation*}
for all $P \in \cP^n$.
Note that $\oPippg$ is an extension of $\oPipp$ to $\cP^n$.
The map $\oPippl \colon \cP^n \to \cK^n_o$ is defined by
\begin{equation*}
	\h{ \oPippl P}{.}^p = \sum_{\substack{v \in \cN(P) \\ \h{P}{v} < 0}} \vol_{n-1}(F(P,v)) |\h{P}{v}|^{1-p} \langle v, . \rangle_+^p
\end{equation*}
for all $P \in \cP^n$.
Note that $\oPippl$ vanishes on $\cP^n_o$.
Similarly we define $\oPipm$, $\oPipmg$ and $\oPipml$ with $\langle v, . \rangle_+^p$ replaced by $\langle v, . \rangle_-^p$,
where $ \langle v, . \rangle_- = \max\{-\langle v, . \rangle,0\}$.

The map $\oDeltapp \colon \cP^n \to \Cp$ is defined by
\begin{align*}
	   \oDeltapp P
	&= \h{\oPippg P}{.}^p - \h{\oPipml P}{.}^p \\
	&= \sum_{\substack{v \in \cN(P) \\ \h{P}{v} \neq 0}} \vol_{n-1}(F(P,v)) \sgn \h{P}{v} |\h{P}{v}|^{1-p} \langle v, . \rangle_{\sgn \h{P}{v}}^p
\end{align*}
for all $P \in \cP^n$.
Note that $\oDeltapp$ is a simple extension of $P \mapsto \h{\oPipp P}{.}^p$ to $\cP^n$.
Similarly we define $\oDeltapm$ by
\begin{equation*}
	   \oDeltapm P = \h{\oPipmg P}{.}^p - \h{\oPippl P}{.}^p
\end{equation*}
for all $P \in \cP^n$.

\begin{lemma}\th\label{operator val}
	Let $V$ be a real vector space and $f \colon \R \times S^{n-1} \to V$ a function.
	Define $\oPhi \colon \cP^n \to V$ by
	\begin{equation*}
		\oPhi P = \sum_{v \in \cN(P)} \vol_{n-1}(F(P,v)) f(\h{P}{v}, v)
	\end{equation*}
	for all $P \in \cP^n$.
	Then $\oPhi$ is a valuation.
\end{lemma}
\begin{proof}
	We need to show that
	\begin{equation*}
		\oPhi (P \cup Q) + \oPhi (P \cap Q) = \oPhi P + \oPhi Q
	\end{equation*}
	for all $P,Q \in \cP^n$ with $P \cup Q \in \cP^n$.
	We distinguish three sets of normal vectors:
	\begin{align*}
		I_1 &:= \{ v \in S^{n-1} : \h{P}{v} < \h{Q}{v} \} ,\\
		I_2 &:= \{ v \in S^{n-1} : \h{P}{v} = \h{Q}{v} \} ,\\
		I_3 &:= \{ v \in S^{n-1} : \h{P}{v} > \h{Q}{v} \} .
	\end{align*}
	For $v \in I_1$ we have $F(P \cup Q,v) = F(Q,v)$, $\h{P \cup Q}{v} = \h{Q}{v}$, $F(P \cap Q,v) = F(P,v)$ and $\h{P \cap Q}{v} = \h{P}{v}$.
	Analogous for $I_3$.
	It follows that the above equation is equivalent to
	\begin{align*}
		&  \sum_{\substack{ v \in \cN(P \cup Q) \\ v \in I_2 }} \vol_{n-1}(F(P \cup Q,v)) f(\h{P \cup Q}{v}, v) \\
		+& \sum_{\substack{ v \in \cN(P \cap Q) \\ v \in I_2 }} \vol_{n-1}(F(P \cap Q,v)) f(\h{P \cap Q}{v}, v) \\
		=& \sum_{\substack{ v \in \cN(P) \\ v \in I_2 }} \vol_{n-1}(F(P,v)) f(\h{P}{v}, v) \\
		+& \sum_{\substack{ v \in \cN(Q) \\ v \in I_2 }} \vol_{n-1}(F(Q,v)) f(\h{Q}{v}, v) .
	\end{align*}
	Note that $f(\h{P \cup Q}{v}, v) = f(\h{P \cap Q}{v}, v) = f(\h{P}{v}, v) = f(\h{Q}{v}, v)$ for $v \in I_2$.
	Furthermore $\cN(P \cup Q) \cup \cN(P \cap Q) = \cN(P) \cup \cN(Q)$.
	Since
	\begin{equation*}
		P \mapsto \vol_{n-1}(F(P,v)) \ , P \in \cP^n
	\end{equation*}
	is a valuation for fixed $v \in S^{n-1}$ (as is easy to verify), this implies the desired result.
\end{proof}

\begin{lemma}\th\label{operator sln contra val}
	The map $\oPippg \colon \cP^n \to \cK^n_o$ is an $\sln$-contravariant $\Lp$-Minkowski valuation.
\end{lemma}
\begin{proof}
	The valuation property follows directly from the definition and \th\ref{operator val} with $f \colon \R \times S^{n-1} \to \Cp$ defined by
	\begin{equation*}
		f(t,v) := \left\{
		\begin{matrix}
				0 & \text{if } t \leq 0 \\
				t^{1-p} \langle v, . \rangle_+^p & \text{if } t > 0\\
		\end{matrix}
		\right.
	\end{equation*}
	for all $t \in \R$ and $v \in S^{n-1}$.
	To show the $\sln$-contravariance let $\phi \in \sln$.
	Note that
	\begin{equation*}
		v \in \cN(P) \Longleftrightarrow \tilde v \in \cN(\phi P)
	\end{equation*}
	with $\tilde v := {\|\phi^{-t} v\|}^{-1} \phi^{-t} v$ and that
	\begin{equation*}
		\vol_{n-1}(F(\phi P, \tilde v)) = {\|\phi^{-t} v\|} \vol_{n-1}(F(P,v)).
	\end{equation*}
	Furthermore
	\begin{align*}
		   f(\h{\phi P}{\tilde v}, \tilde v)
		&= f(\h{P}{\phi^{t} \tilde v}, \tilde v) \\
		&= f({\|\phi^{-t} v\|}^{-1} \h{P}{v}, {\|\phi^{-t} v\|}^{-1} \phi^{-t} v) \\
		&= {\|\phi^{-t} v\|}^{-1} f(\h{P}{v}, \phi^{-t} v) .
	\end{align*}
	Thus, it follows that
	\begin{align*}
		   \vol_{n-1}(F(\phi P, \tilde v)) f(\h{\phi P}{\tilde v}, \tilde v)
		&= \vol_{n-1}(F(P,v)) f(\h{P}{v}, \phi^{-t} v) \\
		&= \vol_{n-1}(F(P,v)) f(\h{P}{v}, v) \circ  \phi^{-1} .
	\end{align*}
	This implies the desired result.
\end{proof}

\begin{remark}\th\label{operator sln contra val rem}
	Analogous to \th\ref{operator sln contra val} we see that $\oPipmg$, $\oPippl$ and $\oPipml$
	are $\sln$-contravariant $\Lp$-Minkowski valuations.
	Furthermore this implies that $\oPipp$ and $\oPipm$ are $\sln$-contravariant $\Lp$-Minkowski valuations.
	A similar result holds for $\oDeltapp$ and $\oDeltapm$.
\end{remark}

\begin{example}\th\label{eval T^n}
	For further reference, we evaluate our valuations at some special polytopes.
	We start with $T^n$.
	The only facet in $T^n$ which does not contain the origin is $F := \conv\{ e_1,\ldots, e_n\}$.
	The outer unit normal at $F$ is $v := \frac 1 {\sqrt{n}} (e_1 + \ldots + e_n)$.
	Note that $\vol_{n-1} F = \frac {\sqrt{n}} {(n-1)!}$ and that $\h{ T^n}{ v} = \frac 1 {\sqrt{n}}$.
	Thus, we have
	\begin{align*}
		   \h{\oPipp T^n}{u}^p
		&= \vol_{n-1}(F) \h{T^n}{v}^{1-p} \langle v, u \rangle_+^p \\
		&= \frac {\sqrt{n}} {(n-1)!} \left( \frac 1 {\sqrt{n}} \right)^{1-p} \left\langle \frac 1 {\sqrt{n}} (e_1 + \ldots + e_n), u \right\rangle_+^p \\
		&= \frac 1 {(n-1)!} \left\langle e_1 + \ldots + e_n, u \right\rangle_+^p
	\end{align*}
	for all $u \in \R^n$.
	Let $i \in \{1,\ldots,n\}$.
	It follows that:
	\begin{align*}
		\h{ \oPipp T^n}{ e_i}^p &= \frac 1 {(n-1)!} ,\\
		\h{ \oPipp T^n}{ -e_i}^p &= 0 .
	\end{align*}
	Similarly we get:
	\begin{align*}
		\h{ \oPipm T^n}{ e_i}^p &= 0 ,\\
		\h{ \oPipm T^n}{ -e_i}^p &= \frac 1 {(n-1)!} .
	\end{align*}

	Next we consider $e_1 + T^n$.
	There are two facets such that the origin is not contained in the affine hull of the facet.
	The first one is $F_1 := e_1 + \conv\{ e_1,\ldots, e_n\}$ with outer unit normal $v_1 := \frac 1 {\sqrt{n}} (e_1 + \ldots + e_n)$.
	Note that $\vol_{n-1} F_1 = \frac {\sqrt{n}} {(n-1)!}$ and that $\h{e_1 + T^n}{v_1} = \frac 2 {\sqrt{n}}$.
	The second one is $F_2 := e_1 + \conv\{ o, e_2,\ldots, e_n\}$ with outer unit normal $v_2 := - e_1$.
	Note that $\vol_{n-1} F_2 = \frac 1 {(n-1)!}$ and that $\h{e_1 + T^n}{v_2} = -1$.
	Thus, we have
	\begin{align*}
		   \h{\oPippg(e_1 + T^n)}{u}^p
		&= \vol_{n-1}(F_1) \h{e_1 + T^n}{v_1}^{1-p} \langle v_1, u \rangle_+^p \\
		&= \frac {\sqrt{n}} {(n-1)!} \left( \frac 2 {\sqrt{n}} \right)^{1-p} \left\langle \frac 1 {\sqrt{n}} (e_1 + \ldots + e_n), u \right\rangle_+^p \\
		&= \frac {2^{1-p}} {(n-1)!} \left\langle e_1 + \ldots + e_n, u \right\rangle_+^p
	\end{align*}
	and
	\begin{align*}
		   \h{\oPippl(e_1 + T^n)}{u}^p
		&= \vol_{n-1}(F_2) |\h{e_1 + T^n}{v_2}|^{1-p} \langle v_2, u \rangle_+^p \\
		&= \frac 1 {(n-1)!} \left\langle - e_1, u \right\rangle_+^p \\
	\end{align*}
	for all $u \in \R^n$.
	It follows that:
	\begin{align*}
		\h{\oPippg (e_1 + T^n)}{e_2-e_1}^p &= 0 ,\\
		\h{\oPippg (e_1 + T^n)}{-e_2+e_1}^p &= 0 ,\\
		\h{\oPippl (e_1 + T^n)}{e_2-e_1}^p &= \frac 1 {(n-1)!} ,\\
		\h{\oPippl (e_1 + T^n)}{-e_2+e_1}^p &= 0 .
	\end{align*}
	Similarly we get:
	\begin{align*}
		\h{\oPipmg (e_1 + T^n)}{e_2-e_1}^p &= 0 ,\\
		\h{\oPipmg (e_1 + T^n)}{-e_2+e_1}^p &= 0 ,\\
		\h{\oPipml (e_1 + T^n)}{e_2-e_1}^p &= 0 ,\\
		\h{\oPipml (e_1 + T^n)}{-e_2+e_1}^p &= \frac 1 {(n-1)!} .
	\end{align*}
\end{example}

\begin{lemma}\th\label{sln contravariance gln}
	Let $\oPhi \colon \cQ^n \rightarrow \Cp$ be $\sln$-contravariant,
	where $\cQ^n$ is either $\cP^n_o$ or $\cP^n$.
	Furthermore let $\phi \in \gln$ with $\det \phi > 0$.
	Then
	\begin{equation*}
		\oPhi(\phi P) = \det(\phi)^{\frac p n} \oPhi \left( \det(\phi)^{\frac 1 n} P \right) \circ \phi^{-1}
	\end{equation*}
	for all $P \in \cQ^n$.
\end{lemma}
\begin{proof}
	Since $\det(\phi)^{-\frac 1 n} \phi \in \sln$,
	this follows directly from the $\sln$-contravariance of $\oPhi$ and the $p$-homogeneity of the functions in $\Cp$.
\end{proof}

	\section{Main Results on $\cP^n_o$}
		%%%%%%%%%%%%%%%%%%%%%%%%%%%%%%%%%%%%%%%%%%%%%%%%%%%%%%%%%%%%%%%%%%%%%%%%%%%%%%%%
\begin{lemma}\th\label{sln contravariant simple}
	Let $\oPhi \colon \cP^n_o \rightarrow \Cp$ be an $\sln$-contravariant valuation.
	If $n \geq 3$, then $\oPhi$ is simple.
\end{lemma}
\begin{proof}
	Let $P \in \cP^n_o$ with $d := \dim P \leq n-1$.
	Because of the $\sln$-contravariance, we can assume without loss of generality that
	\begin{equation*}
		\aff P = \{e_1,\ldots,e_{n-d}\}^\perp = \lin\{e_{n-d+1}, \ldots, e_n\} .
	\end{equation*}
	Define $\phi \in \sln$ by
		$$\phi =
		\begin{pmatrix}
				A & 0 \\
				B & I \\
		\end{pmatrix} ,$$
	where $A \in \R^{(n-d) \times (n-d)}$ is a matrix with $\det A = 1$, $B \in \R^{d \times (n-d)}$ is an arbitrary matrix,
	$0 \in \R^{(n-d) \times d}$ is the zero matrix and $I \in \R^{d \times d}$ is the identity matrix.
	Let $x \in \R^n$.
	Write $x = (x', x'')^t \in \R^{n-d} \times \R^d$ and assume $x' \neq 0$.
	Since $\phi P = P$ and since $\oPhi$ is $\sln$-contravariant, we have
	\begin{equation}\label{sln contravariant simple - phi P}
		\oPhi(P)(x) = \oPhi(\phi P)(x) = \oPhi(P)(\phi^{-1} x) .
	\end{equation}
	A simple calculation yields
	\begin{equation}\label{sln contravariant simple - phi inverse}
		\phi^{-1} x =
		\begin{pmatrix}
				A^{-1} & 0 \\
				-B A^{-1} & I \\
		\end{pmatrix} \cdot
		\begin{pmatrix}
				x' \\
				x'' \\
		\end{pmatrix} =
		\begin{pmatrix}
				A^{-1} x' \\
				-B A^{-1} x' + x'' \\
		\end{pmatrix} .
	\end{equation}

	In the case $d \leq n-2$, we can choose $A$ with $\det A = 1$ such that $A^{-1} x'$ is any nonzero vector.
	After choosing $A$ we can choose $B$ such that $-B A^{-1} x' + x''$ is any vector.
	Combining \eqref{sln contravariant simple - phi P} and \eqref{sln contravariant simple - phi inverse} it follows that
	$\oPhi(P)$ is constant on a dense subset of $\R^n$.
	Because $\oPhi(P)$ is continuous, it must be constant everywhere.
	Since $\oPhi(P)(0) = 0$, we get $\oPhi(P) = 0$.

	Now let $d = n-1$.
	Using \th\ref{sln contravariance gln} and \th\ref{valuation determined simplices} in dimension $n-1$
	it is enough to show that $\oPhi(sT^{n-1}) = 0$ for $s > 0$,
	where $T^{n-1}$ denotes the standard simplex in $e_1^\perp$.
	Analogous to the case $d \leq n-2$ we see that $\oPhi(sT^{n-1})(x) = \oPhi(sT^{n-1})\left( (x',0)^t \right)$.
	Note that $x' \in \R$.
	Because $\oPhi(sT^{n-1})$ is $p$-homogeneous, we only need to show that $\oPhi(sT^{n-1})(\pm e_1) = 0$.

	Let $\lambda \in (0,1)$ and denote by $H_\lambda$ the hyperplane through $o$ with normal vector $\lambda e_2-(1-\lambda)e_3$.
	Since $\oPhi$ is a valuation, we have
	\begin{equation*}
		\oPhi(sT^{n-1}) + \oPhi(sT^{n-1} \cap H_\lambda) = \oPhi(sT^{n-1} \cap H_\lambda^+) + \oPhi(sT^{n-1} \cap H_\lambda^-) ,
	\end{equation*}
	where $H_\lambda^+$ and $H_\lambda^-$ denote the two halfspaces bounded by $H_\lambda$.
	Because $sT^{n-1} \cap H_\lambda$ has dimension $n-2$,
	the case $d \leq n-2$ shows that $\oPhi(sT^{n-1} \cap H_\lambda) = 0$ and we obtain
	\begin{equation}\label{sln contravariant simple - valuation property two}
		\oPhi(sT^{n-1}) = \oPhi(sT^{n-1} \cap H_\lambda^+) + \oPhi(sT^{n-1} \cap H_\lambda^-) .
	\end{equation}
	Define $\phi_\lambda \in \sln$ by
		$$\phi_\lambda e_1 = \frac 1 \lambda e_1 ,\quad
		  \phi_\lambda e_2 = e_2 ,\quad
		  \phi_\lambda e_3 = (1-\lambda)e_2 + \lambda e_3 ,\quad
		  \phi_\lambda e_k = e_k \quad \text{for } 4 \leq k \leq n$$
	and define $\psi_\lambda \in \sln$ by
		$$\psi_\lambda e_1 = \frac 1 {1-\lambda} e_1 ,\quad
		  \psi_\lambda e_2 = (1-\lambda)e_2 + \lambda e_3 ,\quad
		  \psi_\lambda e_3 = e_3 ,\quad
		  \psi_\lambda e_k = e_k \quad \text{for } 4 \leq k \leq n .$$
	Since
	\begin{equation*}
		sT^{n-1} \cap H_\lambda^+ = \phi_\lambda (sT^{n-1}) \quad \text{and} \quad
		sT^{n-1} \cap H_\lambda^- = \psi_\lambda (sT^{n-1}) ,
	\end{equation*}
	relation \eqref{sln contravariant simple - valuation property two} becomes
	\begin{equation*}
		\oPhi(sT^{n-1}) = \oPhi(\phi_\lambda (sT^{n-1})) + \oPhi(\psi_\lambda (sT^{n-1})) .
	\end{equation*}
	We rewrite the last equation at $e_1$ using the $\sln$-contravariance of $\oPhi$ and the $p$-homogeneity of the functions in $\Cp$ as
	\begin{align*}
		\oPhi(sT^{n-1})(e_1)
		&=
		\oPhi(\phi_\lambda (sT^{n-1}))(e_1) + \oPhi(\psi_\lambda (sT^{n-1}))(e_1) \\
		&=
		\oPhi(sT^{n-1})(\phi_\lambda^{-1} e_1) + \oPhi(sT^{n-1})(\psi_\lambda^{-1} e_1) \\
		&=
		\oPhi(sT^{n-1})(\lambda e_1) + \oPhi(sT^{n-1})((1-\lambda) e_1) \\
		&=
		\lambda^p \oPhi(sT^{n-1})(e_1) + (1-\lambda)^p \oPhi(sT^{n-1})(e_1) \\
		&=
		\left( \lambda^p + (1-\lambda)^p \right) \oPhi(sT^{n-1})(e_1) .
	\end{align*}
	Since $p > 1$, the resulting equation can only hold if $\oPhi(sT^{n-1})(e_1) = 0$.
	Similarly we see that $\oPhi(sT^{n-1})(-e_1) = 0$.
\end{proof}

\begin{theorem}\th\label{class sln contra}
	Let $\oPhi \colon \cP^n_o \rightarrow \Cp$ be an $\sln$-contravariant valuation.
	Assume further that for every $y \in \R^n$ there exists a bounded open interval $I_y \subseteq (0,+\infty)$
	such that $\{\oPhi(sT^n)(y): s \in I_y\}$ is not dense in $\R$.
	If $n \geq 3$, then there exist constants $c_1, c_2 \in \R$ such that
	\begin{equation*}
		\oPhi P = c_1 \h{\oPipp P}{.}^p + c_2 \h{\oPipm P}{.}^p
	\end{equation*}
	for all $P \in \cP^n_o$.
\end{theorem}
\begin{proof}
	By \th\ref{sln contravariance gln} and \th\ref{valuation determined simplices}, it is sufficient to proof that
	\begin{equation}\label{class sln contra - simplices}
		\oPhi( s T^n ) = c_1 \h{\oPipp( s T^n )}{.}^p + c_2 \h{\oPipm( s T^n )}{.}^p
	\end{equation}
	for $s > 0$.

	\noindent\textbf{1. Functional Equation: }
	Let $\lambda \in (0,1)$ and denote by $H_\lambda$ the hyperplane through $o$ with normal vector $\lambda e_1-(1-\lambda)e_2$.
	Since $\oPhi$ is a valuation, we have
	\begin{equation*}
		\oPhi(sT^n) + \oPhi(sT^n \cap H_\lambda) = \oPhi(sT^n \cap H_\lambda^+) + \oPhi(sT^n \cap H_\lambda^-) .
	\end{equation*}
	Because $sT^n \cap H_\lambda$ has dimension $n-1$,
	\th\ref{sln contravariant simple} shows that $\oPhi(sT^n \cap H_\lambda) = 0$ and we obtain
	\begin{equation}\label{class sln contra - val property two}
		\oPhi(sT^n) = \oPhi(sT^n \cap H_\lambda^+) + \oPhi(sT^n \cap H_\lambda^-) .
	\end{equation}
	Define $\phi_\lambda \in \gln$ by
		$$\phi_\lambda e_1 = e_1 ,\quad
		  \phi_\lambda e_2 = (1-\lambda)e_1 + \lambda e_2 ,\quad
		  \phi_\lambda e_k = e_k \quad \text{for } 3 \leq k \leq n$$
	and define $\psi_\lambda \in \gln$ by
		$$\psi_\lambda e_1 = (1-\lambda)e_1 + \lambda e_2 ,\quad
		  \psi_\lambda e_2 = e_2 ,\quad
		  \psi_\lambda e_k = e_k \quad \text{for } 3 \leq k \leq n .$$
	Note that
	\begin{equation}\label{class sln contra - det}
		\det(\phi_\lambda) = \lambda \quad \text{and} \quad \det(\psi_\lambda) = 1-\lambda .
	\end{equation}
	Since
		$$T^n \cap H_\lambda^+ = \phi_\lambda T^n \quad \text{and} \quad
		  T^n \cap H_\lambda^- = \psi_\lambda T^n ,$$
	relation \eqref{class sln contra - val property two} becomes
	\begin{equation*}
		\oPhi(sT^n) = \oPhi(\phi_\lambda sT^n) + \oPhi(\psi_\lambda sT^n) .
	\end{equation*}
	Using \th\ref{sln contravariance gln} and \eqref{class sln contra - det}, we can rewrite the last equation as
	\begin{equation}\label{class sln contra - functional equation}
		\oPhi(sT^n)(x) =
		\lambda^{\frac p n} \oPhi\left( \lambda^{\frac 1 n} sT^n \right)(\phi_\lambda^{-1} x)
		+ (1-\lambda)^{\frac p n} \oPhi\left( (1-\lambda)^{\frac 1 n} sT^n \right)(\psi_\lambda^{-1} x)
	\end{equation}
	for all $x \in \R^n$.

	\noindent\textbf{2. Homogeneity: }
	For $y \in \{e_1,e_2\}^\perp$ this becomes
	\begin{equation*}
		\oPhi(sT^n)(y) =
		\lambda^{\frac p n} \oPhi\left( \lambda^{\frac 1 n} sT^n \right)(y) + (1-\lambda)^{\frac p n} \oPhi\left( (1-\lambda)^{\frac 1 n} sT^n \right)(y) .
	\end{equation*}
	Define $g(s) = \oPhi\left( s^{\frac 1 n} T^n \right)(y)$, then we have
	\begin{equation*}
		g(s) = \lambda^{\frac p n} g(\lambda s) + (1-\lambda)^{\frac p n} g((1-\lambda) s) .
	\end{equation*}
	Let $a,b > 0$.
	Setting $s = a+b$ and $\lambda = \frac {a} {a+b}$ we obtain:
	\begin{align*}
		g(a+b) &= \left( \frac {a} {a+b} \right)^{\frac p n} g(a) + \left( \frac {b} {a+b} \right)^{\frac p n} g(b) ,\\
		(a+b)^{\frac p n}g(a+b) &= a^{\frac p n} g(a) + b^{\frac p n} g(b) .
	\end{align*}
	Thus, $s \mapsto s^{\frac p n} g(s)$ solves Cauchy's functional equation for $s > 0$.
	By the assumption that there is a bounded open interval $I_y$ such that $g(I_y)$ is not dense in $\R$,
	it follows that $s \mapsto s^{\frac p n} g(s)$ is linear.
	This implies $s^{\frac p n} g(s) = s g(1)$ and hence $g(s) = s^{1-\frac p n} g(1)$.
	The definition of $g$ therefore yields
	\begin{equation*}
		\oPhi(s T^n)(y) = g(s^n) = s^{n-p} g(1) = s^{n-p} \oPhi(T^n)(y) .
	\end{equation*}
	Since $n \geq 3$ and since we can repeat the above argument for any two standard basis vectors, we obtain
	\begin{equation}\label{class sln contra - homogeneity}
		\oPhi(s T^n)(\pm e_i) = s^{n-p} \oPhi(T^n)(\pm e_i) \quad \text{for } i = 1,\ldots,n .
	\end{equation}

	\noindent\textbf{3. Constants: }
	Let $i \in \{1,\ldots,n\}$.
	Since $n \geq 3$, we can find a permutation of the coordinates $\phi \in \sln$ such that $\phi e_1 = e_i$.
	It follows that
	\begin{equation}\label{class sln contra - values on standard basis}
		\oPhi(T^n)(e_i) = \oPhi(T^n)(\phi^{-1} e_1) = \oPhi( \phi T^n)(e_1) = \oPhi(T^n)(e_1) .
	\end{equation}
	Similarly we get $\oPhi(T^n)(-e_i) = \oPhi(T^n)(-e_1)$.
	Set
	\begin{equation}\label{class sln contra - constants}
		c_1 = (n-1)! \oPhi(T^n)(e_1) \quad \text{and} \quad c_2 = (n-1)! \oPhi(T^n)(-e_1) .
	\end{equation}
	
	\noindent\textbf{4. Induction: }
	We are now going to show by induction on the number $m$ of coordinates of $x$ not equal to zero that
	\begin{equation}\label{class sln contra - induction}
		\oPhi( s T^n )(x) = c_1 \h{\oPipp( sT^n )}{x}^p + c_2 \h{\oPipm( s T^n)}{x}^p
	\end{equation}
	for $s > 0$ and for all $x \in \R^n$.
	Note that since $P \mapsto c_1 \h{\oPipp(P)}{.}^p + c_2 \h{\oPipm(P)}{.}^p$ satisfies the assumptions of the theorem it also satisfies
	\eqref{class sln contra - functional equation} and \eqref{class sln contra - homogeneity}.
	
	The case $m = 0$ is trivial.
	The case $m = 1$ is also easy to verify with \eqref{class sln contra - homogeneity},
	\eqref{class sln contra - values on standard basis}, \eqref{class sln contra - constants} and \th\ref{eval T^n}.
	Let $m \geq 2$ and write $x = (x_1, \ldots, x_n)^t$.
	Assume without loss of generality that $x_1, x_2 \neq 0$ and that $|x_1| \leq |x_2|$.
	Since the functions in $\Cp$ are continuous, we can further assume that $|x_1| < |x_2|$.

	First consider the case where $x_1,x_2$ have the same sign.
	Set $\lambda = \frac {x_2} {x_1 + x_2} \in (0,1)$ and calculate
	\begin{align*}
		&  \phi_\lambda \left( (x_1 + x_2) e_2 + x_3 e_3 + \ldots + x_n e_n \right) \\
		&= (x_1 + x_2) (1-\lambda) e_1 + (x_1 + x_2) \lambda e_2 + x_3 e_3 + \ldots + x_n e_n \\
		&= x_1 e_1 + x_2 e_2 + x_3 e_3 + \ldots + x_n e_n \\
		&= x
	\end{align*}
	or equivalently
	\begin{equation*}
		\phi_\lambda^{-1} x = \left( (x_1 + x_2) e_2 + x_3 e_3 + \ldots + x_n e_n \right) .
	\end{equation*}
	Similarly we calculate
	\begin{equation*}
		\psi_\lambda^{-1} x = \left( (x_1 + x_2) e_1 + x_3 e_3 + \ldots + x_n e_n \right) .
	\end{equation*}
	Using \eqref{class sln contra - functional equation} and the induction hypotheses yields the desired result.
	
	Now consider the case where $x_1,x_2$ have different signs.
	Set $\lambda = 1+ \frac {x_1} {x_2} \in (0,1)$ and calculate
	\begin{align*}
		\phi_\lambda x
		&= x_1 e_1 + x_2 (1-\lambda) e_1 + x_2 \lambda e_2 + x_3 e_3 + \ldots + x_n e_n \\
		&= (x_1 + x_2) e_2 + x_3 e_3 + \ldots + x_n e_n .
	\end{align*}
	Similarly we calculate
	\begin{equation*}
		\psi_\lambda^{-1} \phi_\lambda x = (x_1 + x_2) e_2 + x_3 e_3 + \ldots + x_n e_n .
	\end{equation*}
	Using \eqref{class sln contra - functional equation} with $x$ replaced by $\phi_\lambda x$ and using the induction hypotheses yields the desired result.

	This completes the induction and proves \eqref{class sln contra - induction} and thus \eqref{class sln contra - simplices}.
\end{proof}

\begin{remark}\th\label{further assumption}
	The additional assumption on the boundedness in the last theorem is only used for the vectors of the standard basis
	and their inverses in the proof.
	By reasoning similar to step three in the proof of \th\ref{class sln contra} it is enough
	to have this assumption for only one standard basis vector and its reflection at the origin.
\end{remark}

\begin{corollary}\th\label{class sln contra mink}
	Let $\oPhi \colon \cP^n_o \rightarrow \cK^n_o$ be an $\sln$-contravariant $\Lp$-Minkowski valuation.
	If $n \geq 3$, then there exist constants $c_1, c_2 \geq 0$ such that
	\begin{equation*}
		\oPhi P = c_1 \oPipp P \pp c_2 \oPipm P
	\end{equation*}
	for all $P \in \cP^n_o$.
\end{corollary}
\begin{proof}
	Since $\oPhi(sT^n) \in \cK^n_o$, we have $\h{\oPhi(sT^n)}{y} \geq 0$ for $s > 0$ and for all $y \in \R^n$.
	Therefore the map
	\begin{equation*}
		P \mapsto \h{\oPhi P}{.}^p, \quad P \in \cP^n_o
	\end{equation*}
	satisfies the assumptions of \th\ref{class sln contra}.
	Thus there exist constants $d_1, d_2 \in \R$ such that
	\begin{equation*}
		\h{\oPhi P}{.}^p = d_1 \h{\oPipp P}{.}^p + d_2 \h{\oPipm P}{.}^p .
	\end{equation*}
	According to \eqref{class sln contra - constants} these constants are given by:
	\begin{align*}
		d_1 &= (n-1)! \h{\oPhi T^n}{e_1}^p ,\\
		d_2 &= (n-1)! \h{\oPhi T^n}{-e_1}^p .
	\end{align*}
	It follows that $d_1, d_2 \geq 0$.
	Defining $c_1 = \sqrt[p]{d_1}$ and $c_2 = \sqrt[p]{d_2}$ completes the proof.
\end{proof}

The first main theorem from the introduction now follows from \th\ref{class sln contra mink}
and the fact that $\oPipp$ and $\oPipm$ have the desired properties.

	\section{Main Results on $\cP^n$}
		%%%%%%%%%%%%%%%%%%%%%%%%%%%%%%%%%%%%%%%%%%%%%%%%%%%%%%%%%%%%%%%%%%%%%%%%%%%%%%%%
\begin{lemma}\th\label{sln contravariant Pn zero}
	Let $\oPhi \colon \cP^n \rightarrow \Cp$ be an $\sln$-contravariant valuation.
	If $n \geq 3$, then $\oPhi P = 0$ for all $P \in \cP^n$ with $\dim P \leq n-2$ and for all $P \in \cP^n$ with $\dim P = n-1$ and $0 \in \aff P$.
\end{lemma}
\begin{proof}
	Since $\oPhi$ is $\sln$-contravariant, it is sufficient to prove $\oPhi P = 0$ for all $P \in \cP^n$ with $P \subseteq e_1^\perp$.
	Let $d := \dim P$.
	There are several different cases:
	\begin{enumerate}[(i)]
		\item $d \leq n-3$
		\item $d = n-2$ and $0 \in \aff P$
		\item $d = n-2$ and $0 \not\in \aff P$
		\item $d = n-1$ .
	\end{enumerate}

	In the cases (i) and (ii) we can use the $\sln$-contravariance of $\oPhi$ to assume that $P \subseteq \{e_1,e_2\}^\perp$.
	Here we can use the same reasoning as for the case $d = n-2$ in the proof of \th\ref{sln contravariant simple}.

	Now consider the case (iii).
	Using \th\ref{sln contravariance gln} and \th\ref{valuation determined simplices Pn} in $\aff P$
	it is enough to show that $\oPhi(sT^{n-2}) = 0$ for $s > 0$,
	where $T^{n-2} := \conv\{e_2,\ldots,e_n\}$.
	Here we can use the same reasoning as for the case $d = n-1$ in the proof of \th\ref{sln contravariant simple}.

	Finally consider the case (iv).
	\th\ref{sln contravariant simple} shows that $\oPhi P = 0$ if $0 \in P$.
	Using the cases (i)-(iii) and \th\ref{simple determined cPno} in dimension $n-1$ gives the desired result.
\end{proof}

\begin{theorem}\th\label{class sln contra Pn simple}
	Let $\oPhi \colon \cP^n \rightarrow \Cp$ be a simple $\sln$-contravariant valuation.
	Assume further that for every $y \in \R^n$ there exists a bounded open interval $I_y \subseteq (0,+\infty)$
	such that $\{\oPhi(sT^n)(y): s \in I_y\}$ is not dense in $\R$.
	If $n \geq 3$, then there exist constants $c_1, c_2 \in \R$ such that
	\begin{equation*}
		\oPhi P = c_1 \oDeltapp P + c_2 \oDeltapm P
	\end{equation*}
	for all $P \in \cP^n$.
\end{theorem}
\begin{proof}
	\th\ref{class sln contra} implies that there are constants $c_1, c_2 \in \R$ such that
	\begin{equation*}
		\oPhi P = c_1 \h{\oPipp P}{.}^p + c_2 \h{\oPipm P}{.}^p
	\end{equation*}
	for all $P \in \cP^n_o$.
	Since $c_1 \oDeltapp + c_2 \oDeltapm$ is a simple $\sln$-contravariant valuation which coincides with $\oPhi$ on $\cP^n_o$,
	the assertion follows from \th\ref{simple determined cPno}.
\end{proof}

\begin{theorem}\th\label{class sln contra Pn}
	Let $\oPhi \colon \cP^n \rightarrow \Cp$ be an $\sln$-contravariant valuation.
	Assume further that for every $y \in \R^n$ there exists a bounded open interval $I_y \subseteq (0,+\infty)$
	such that $\{\oPhi(sT^n)(y): s \in I_y\}$ is not dense in $\R$
	and that for every $y \in \R^n$ there exists a bounded open interval $J_y \subseteq (0,+\infty)$
	such that $\{\oPhi(sT^{n-1})(y): s \in J_y\}$ is not dense in $\R$,
	where $T^{n-1} := \conv\{e_1,\ldots,e_n\}$.
	If $n \geq 3$, then there exist constants $c_1, c_2, c_3, c_4 \in \R$ such that
	\begin{equation*}
		\oPhi P = c_1 \h{\oPippg P}{.}^p + c_2 \h{\oPipmg P}{.}^p + c_3 \h{\oPippl P}{.}^p + c_4 \h{\oPipml P}{.}^p
	\end{equation*}
	for all $P \in \cP^n$.
\end{theorem}
\begin{proof}
	By replacing $T^n$ with $T^{n-1}$ and \th\ref{sln contravariant simple} with \th\ref{sln contravariant Pn zero}
	in the proof of \th\ref{class sln contra} we see that there exist constants $d_3, d_4$ such that
	\begin{equation*}
		\oPhi( s T^{n-1} ) = d_3 \h{\oPippg( s T^{n-1} )}{.}^p + d_4 \h{\oPipmg( s T^{n-1} )}{.}^p
	\end{equation*}
	for $s > 0$.
	The constants are given by
	\begin{equation*}
		d_3 = (n-1)! \oPhi(T^{n-1})(e_1) \quad \text{and} \quad d_4 = (n-1)! \oPhi(T^{n-1})(-e_1) .
	\end{equation*}
	Define $\oPsi \colon \cP^n \rightarrow \Cp$ by
	\begin{equation*}
		\oPsi P = \oPhi P - d_3 \h{\oPippg P}{.}^p - d_4 \h{\oPipmg P}{.}^p
	\end{equation*}
	for all $P \in \cP^n$.
	Note that $\oPsi$ is an $\sln$-contravariant valuation.
	Using \th\ref{sln contravariant Pn zero}, $\oPsi( sT^{n-1}) = 0$, the $\sln$-contravariance of $\oPsi$ and \th\ref{incl excl Pn} we see that $\oPsi$ is simple.
	Now \th\ref{class sln contra Pn simple} implies that there exist constants $d_1, d_2 \in \R$ such that
	\begin{equation*}
		\oPsi P = d_1 \oDeltapp P + d_2 \oDeltapm P
	\end{equation*}
	for all $P \in \cP^n$.
	The definition of $\oPsi$ yields
	\begin{align*}
		\oPhi P
		&= \oPsi P + d_3 \h{\oPippg P}{.}^p + d_4 \h{\oPipmg P}{.}^p \\
		&= d_1 \oDeltapp P + d_2 \oDeltapm P + d_3 \h{\oPippg P}{.}^p + d_4 \h{\oPipmg P}{.}^p \\
		&= d_1 \left( \h{\oPippg P}{.}^p - \h{\oPipml P}{.}^p \right) + d_2  \left( \h{\oPipmg P}{.}^p - \h{\oPippl P}{.}^p \right) \\
		&\hspace{1em}+ d_3 \h{\oPippg P}{.}^p + d_4 \h{\oPipmg P}{.}^p \\
		&= (d_1+d_3) \h{\oPippg P}{.}^p + (d_2+d_4) \h{\oPipmg P}{.}^p \\
		&\hspace{1em}- d_2 \h{\oPippl P}{.}^p - d_1 \h{\oPipml P}{.}^p
	\end{align*}
	for all $P \in \cP^n$.
	Define $c_1 = d_1+d_3$, $c_2 = d_2+d_4$, $c_3 = -d_2$ and $c_4 = -d_1$ to complete the proof.
\end{proof}

\begin{remark}
	Similar to the $\cP^n_o$ case, the additional assumptions on the boundedness in \th\ref{class sln contra Pn simple}
	and \th\ref{class sln contra Pn} just have to hold for only one standard basis vector and its reflection at the origin each.
\end{remark}

\begin{corollary}\th\label{class sln contra mink Pn}
	Let $\oPhi \colon \cP^n \rightarrow \cK^n_o$ be an $\sln$-contravariant $\Lp$-Minkowski valuation.
	If $n \geq 3$, then there exist constants $c_1, c_2, c_3, c_4 \geq 0$ such that
	\begin{equation*}
		\oPhi P = c_1 \oPippg P \pp c_2 \oPipmg P \pp c_3 \oPippl P \pp c_4 \oPipml P
	\end{equation*}
	for all $P \in \cP^n$.
\end{corollary}
\begin{proof}
	Since $\oPhi P \in \cK^n_o$, we have $\h{\oPhi P}{y} \geq 0$ for all $P \in \cP^n$ and $y \in \R^n$.
	Therefore the map
	\begin{equation*}
		P \mapsto \h{\oPhi P}{.}^p, \quad P \in \cP^n
	\end{equation*}
	satisfies the assumptions of \th\ref{class sln contra Pn}.
	Thus there exist constants $d_1, d_2, d_3, d_4 \in \R$ such that
	\begin{equation*}
		\h{\oPhi P}{.}^p = d_1 \h{\oPippg P}{.}^p + d_2 \h{\oPipmg P}{.}^p + d_3 \h{\oPippl P}{.}^p + d_4 \h{\oPipml P}{.}^p
	\end{equation*}
	for all $P \in \cP^n$.
	Using \th\ref{eval T^n} we obtain:
	\begin{align*}
		d_1 &= (n-1)! \h{\oPhi T^n}{e_1}^p ,\\
		d_2 &= (n-1)! \h{\oPhi T^n}{-e_1}^p ,\\
		d_3 &= (n-1)! \h{\oPhi (e_1 + T^n)}{e_2-e_1}^p ,\\
		d_4 &= (n-1)! \h{\oPhi (e_1 + T^n)}{e_1-e_2}^p .
	\end{align*}
	It follows that $d_1, d_2, d_3, d_4 \geq 0$.
	Defining $c_1 = \sqrt[p]{d_1}, c_2 = \sqrt[p]{d_2}, c_3 = \sqrt[p]{d_3}$ and $c_4 = \sqrt[p]{d_4}$ completes the proof.
\end{proof}

The second main theorem from the introduction now follows from \th\ref{class sln contra mink Pn}
and the fact that $\oPippg$, $\oPipmg$, $\oPippl$ and $\oPipml$ have the desired properties.

\begin{remark}
	\th\ref{class sln contra mink Pn} implies \th\ref{class sln contra mink}.
	To see this, note that we can extend an $\sln$-contravariant $\Lp$-Minkowski valuation $\oPhi$ on $\cP^n_o$ to $\cP^n$ by
	\begin{equation*}
		\oPhi P = \oPhi (P_o)
	\end{equation*}
	for all $P \in \cP^n$, where $P_o := \conv(\{0\} \cup P)$.
\end{remark}

\end{document}